\documentclass{amsart}

\usepackage{amssymb}
\usepackage[utf8]{inputenc}
\usepackage{tikz}
\usepackage{tikz-cd}
\usepackage{amsmath, amsthm, amssymb, amsfonts}
\usepackage{hyperref}
\usepackage{mathtools}
\usepackage[retainorgcmds]{IEEEtrantools}
\usepackage{latexsym}
\usepackage{enumerate}
\usepackage{nomencl}
\usepackage{nameref}
\usepackage{chngcntr}

\makeatletter
\let\orgdescriptionlabel\descriptionlabel
\renewcommand*{\descriptionlabel}[1]{%
  \let\orglabel\label
  \let\label\@gobble
  \phantomsection
  \edef\@currentlabel{#1}%
  \let\label\orglabel
  \orgdescriptionlabel{#1}%
}
\makeatother

\newtheorem*{thma}{Theorem}
\newtheorem*{propa}{Proposition}

\newtheorem{thm}{Theorem}[section]
\newtheorem{lem}[thm]{Lemma}
\newtheorem{prop}[thm]{Proposition}
\newtheorem{cor}[thm]{Corollary}

\theoremstyle{definition}
\newtheorem{defi}[thm]{Definition}
\newtheorem{ex}[thm]{Example}

\theoremstyle{remark}
\newtheorem{rem}[thm]{Remark}

\newcommand{\op}{\normalfont{^{op}}}

\newcommand{\Mod}[1]{\operatorname{mod} #1}

\newcommand{\Hom}[3]{\operatorname{Hom}_{#1}\left(#2,#3\right) }

\newcommand{\End}[2]{\operatorname{End}_{#1}\left(#2\right) }

\newcommand{\Soc}[1]{\operatorname{Soc}#1}
\newcommand{\Top}[1]{\operatorname{Top}#1}

\newcommand{\Rad}[1]{\operatorname{Rad} #1}
\newcommand{\RAD}[2]{\operatorname{Rad}^{#1} #2}
\newcommand{\SOC}[2]{\operatorname{Soc}_{#1} #2}
\newcommand{\Ima}[1]{\operatorname{Im}#1}
\newcommand{\Ker}[1]{\operatorname{Ker}#1}

\newcommand{\Tr}[2]{\operatorname{Tr}\left(#1,#2\right)}

\newcommand{\Rej}[2]{\operatorname{Rej}\left(#1,#2\right)}
\newcommand{\Ext}[4]{\operatorname{Ext}_{#1}^{#2}\left(#3,#4 \right)}

\newcommand{\Z}{\mathbb{Z}}
\newcommand{\Znn}{\Z_{\geq 0}}
\newcommand{\Zp}{\Z_{> 0}}

\newcommand{\B}[1]{\left(#1\right)}

\newcommand{\dssl}[1]{\operatorname{\Delta.ssl}#1}

\newcommand{\Add}[1]{\operatorname{add}#1}

\newcommand{\LL}[1]{\operatorname{LL}(#1)}

\begin{document}

\title[$\Delta$-filtrations and projective resolutions for the ADR algebra]{$\Delta$-filtrations and projective resolutions for the Auslander--Dlab--Ringel algebra}
\author{Teresa Conde}
\address{Institute of Algebra and Number Theory, University of Stuttgart\\ Pfaffenwaldring 57, 70569 Stuttgart, Germany}
\email{\href{mailto:tconde@mathematik.uni-stuttgart.de}{\nolinkurl{tconde@mathematik.uni-stuttgart.de}}}
%\author{Karin Erdmann}
%\address{Mathematical Institute,
%University of Oxford\\ Radcliffe Observatory Quarter,
%OX2 6GG, United Kingdom}
%\email{\href{mailto:erdmann@maths.ox.ac.uk}{\nolinkurl{erdmann@maths.ox.ac.uk}}}
\thanks{Most of this work is contained in the author's Ph.D.~thesis. This was supported by the grant SFRH/BD/84060/2012 of Funda\c{c}\~ao para a Ci\^encia e a Tecnologia, Portugal. The author would like to express her gratitude to Stephen Donkin, Ph.D.~examiner, for the simplified version of the proofs of Lemma \ref{lem:ultrastronglyses} and Corollary \ref{cor:ultrastronglyqhdeltass} included in this article. In addition, the author would like to thank her Ph.D.~supervisor, Karin Erdmann, for many useful comments.}
\date{\today}

\begin{abstract}
The ADR algebra $R_A$ of an Artin algebra $A$ is a right ultra strongly quasihereditary algebra (RUSQ algebra). In this paper we study the $\Delta$-filtrations of modules over RUSQ algebras and determine the projective covers of a certain class of $R_A$-modules. As an application, we give a counterexample to a claim by Auslander--Platzeck--Todorov, concerning projective resolutions over the ADR algebra.
\end{abstract}

\maketitle

\section{Introduction}
\label{sec:overviewpart2}
It is natural to ask whether there exist ``Schur algebras" for arbitrary Artin algebras. That is, given an Artin algebra $A$, we would like to have an $A$-module whose endomorphism algebra is quasihereditary, so that it has finite global dimension and a highest weight theory. Such modules do exist. A suitable candidate was introduced by Auslander in \cite{MR0349747}. He showed that the endomorphism algebra of
\[G=\bigoplus_{i = 1}^{\LL{A}} A/\RAD{i}{A}\]
has finite global dimension (here $\LL{A}$ denotes the Loewy length of $A$). Subsequently, Dlab and Ringel proved that this endomorphism algebra is actually a quasihereditary algebra (\cite{MR943793}). For practical purposes one considers the basic version of $\End{A}{G}\op$ instead. We denote this `Schur-like' endomorphism algebra by $R_A$ and call it the Auslander--Dlab--Ringel algebra (ADR algebra) of $A$. The original algebra $A$ is then Morita equivalent to $\xi R_A \xi$ for an idempotent $\xi$ in $R_A$, and this is also analogous to the situation of symmetric groups and Schur algebras.

It was recently proved in \cite{MR3510398} that the ADR algebra has a particularly neat quasihereditary structure. The ADR algebra is not only right strongly quasihereditary in the sense of Ringel (\cite{RingelIyama}); $R_A$ is actually a right ultra strongly quasihereditary algebra (RUSQ algebra) as defined in \cite{MR3510398} (see also \cite[Chapter 2]{thesis}). The ADR algebra is not the only strongly quasihereditary algebra arising from a module theoretical context. Other examples of strongly quasihereditary algebras include: the Auslander algebras, associated to algebras of finite type; the endomorphism algebras constructed by Iyama, used in his famous proof of the finiteness of the representation dimension of Artin algebras (\cite{iyama}); certain cluster-tilted algebras studied by Gei{\ss}--Leclerc--Schr{\"o}er (\cite{KacMoodyAdv}, \cite{GeissLeclercSchroerArxiv}), Buan--Iyama--Reiten--Scott (\cite{MR2521253}) and Iyama--Reiten (\cite{IyamaReiten}). The cluster-tilted algebras in \cite{MR2521253}, \cite{KacMoodyAdv}, \cite{GeissLeclercSchroerArxiv} and \cite{IyamaReiten} are actually RUSQ up to dualisation, as implicitly proved in \cite{KacMoodyAdv} (see also \cite{thesis}).

This paper complements the investigation on RUSQ algebras and on ADR algebras initiated in \cite{MR3510398}. We start by studying the $\Delta$-filtrations of modules over RUSQ algebras. In Section \ref{sec:deltass}, we show that the RUSQ algebras satisfy the following key property: every submodule of a direct sum of standard modules is still a direct sum of standard modules. This has several consequences. In particular, it gives rise to special (uniquely determined) filtrations of $\Delta$-good modules over RUSQ algebras, called $\Delta$-semisimple filtrations. These can be described recursively as follows. Given a $\Delta$-good module $N$, let $\delta_1(N)$ be the largest submodule of $N$ which is a direct sum of standard modules, and for $i\geq 1$ define $\delta_{i+1}(N)$ as the module satisfying the identity $\delta_{i+1}(N)/\delta_{i} (N) = \delta_1\B{N/\delta_{i} (N)}$. 

The ideas in Section \ref{sec:deltass} are then applied to the ADR algebra. As a main contribution of Section \ref{sec:manuela}, we prove the following (which corresponds to Lemma \ref{lem:soc0} and Theorem~\ref{thm:socdelta}).
\begin{thma}
Let $M$ be in $\Mod{A}$. Then $N=\Hom{A}{G}{M}$ lies in $\mathcal{F}\B{\Delta}$ and the socle series of $M$ determines the $\Delta$-semisimple filtration of $N$. Formally,
\[
\delta_m \B{N} = \Hom{A}{G}{\SOC{m}{M}},
\]
for all $m$. The factors of the $\Delta$-semisimple filtration of $N$ only depend on the factors of the socle series of $M$ and on the Loewy length of the projective indecomposable $A$-modules.
\end{thma}

Next, we describe the right minimal $\Add{G}$-approximations of rigid modules in $\Mod{A}$, or equivalently, the projective covers of the $R_A$-modules $\Hom{A}{G}{M}$, with $M$ rigid. Recall that a module is said to be rigid if its radical series coincides with its socle series. We prove the following theorem in Section \ref{sec:thma}.
\begin{thma}
%\label{customthm:A}
Let $M$ be a rigid module in $\Mod{A}$, with Loewy length $m$. Then the projective cover of $M$ in $\Mod{(A/ \RAD{m}{A})}$ is a right minimal $\Add{G}$-approximation of $M$.
\end{thma}

This simple yet useful result, combined with the conclusions in Section \ref{sec:manuela}, is then used to provide a counterexample to a claim by Auslander, Platzeck and Todorov in \cite[§7]{MR1052903}, about the projective resolutions of modules over the ADR algebra, for which no proof was given. To be precise, we show the following.
\begin{propa}
ADR algebras need not satisfy the descending Loewy length condition on projective resolutions.
\end{propa}

\section{Preliminaries}
In this section we introduce the language of preradicals and give some background on quasihereditary algebras, RUSQ algebras and the ADR algebra.

Throughout this paper the letters $B$ and $A$ shall denote arbitrary Artin algebras over some commutative artinian ring $C$. All the modules will be finitely generated left modules. The notation $\Mod{B}$ will be used for the category of (finitely generated) $B$-modules. Given a class of $B$-modules $\Theta$, let $\Add{\Theta}$ be the full subcategory of $\Mod{B}$ consisting of all modules isomorphic to a summand of a direct sum of modules in $\Theta$. The additive closure of a single module $M$ is denoted by $\Add{M}$.
\subsection{Preradicals}
\label{sec:preradicals}
Preradicals generalise the classic notions of radical and socle of a module. The results and definitions stated in this section are elementary and most of the proofs may be found in \cite{MR655412}, \cite[Chapter 2]{MR2253001} and \cite[Chapter VI]{MR0389953} (see also \cite[Section $1.3$]{thesis}). 

\subsubsection{Definition and first properties}
\label{subsec:defisandfirst}
\begin{defi}
\label{defi:preradical}
A \emph{preradical}\index{preradical} $\tau$ in $\Mod{B}$ is a subfunctor of the identity functor $1_{\Mod{B}}$, i.e., $\tau$ assigns to each module $M$ a submodule $\tau\B{M}$, such that each morphism $f:M\longrightarrow N$ induces a morphism $\tau\B{f}: \tau\B{M} \longrightarrow \tau\B{N}$ given by restriction.
\end{defi}

A submodule $N$ of a $B$-module $M$ is called a \emph{characteristic submodule} of $M$ if $f\B{N} \subseteq N$, for every $f$ in $\End{B}{M}$. By definition, it is clear that the module $\tau\B{M}$ is a characteristic submodule of $M$, for every preradical $\tau$ and for every module $M$. It is also evident that every preradical is an additive functor which preserves monics.

To each preradical $\tau$ we may associate the functor%\label{eq:factorofapreradical}
\[
1/ \tau: \Mod{B} \longrightarrow \Mod{B},
\]
which maps $M$ to $M/\tau\B{M}$. Note that the functor $1/ \tau$ preserves epics.

\begin{ex}
\label{ex:tracereject}
For any class $\Theta$ of $B$-modules, the operators defined by
\begin{gather*}
\Tr{\Theta}{M}:=\sum_{f: \, f\in \Hom{B}{U}{M},\, U \in \Theta} \Ima{f},\\
\Rej{M}{\Theta}:=\bigcap_{f: \, f\in \Hom{B}{M}{U},\, U \in \Theta} \Ker{f},
\end{gather*}
for $M$ in $\Mod{B}$, are preradicals in $\Mod{B}$. The module $\Tr{\Theta}{M}$, called the \emph{trace of $\Theta$ in $M$}, is the largest submodule of $M$ generated by $\Theta$. Symmetrically, $\Rej{M}{\Theta}$, the \emph{reject of $\Theta$ in $M$}, is the submodule $N$ of $M$ such that $M/N$ is the largest factor module of $M$ cogenerated by $\Theta$. If $\varepsilon$ is a complete set of simple $B$-modules, then $\Tr{\varepsilon}{-}= \Soc{(-)}$ and $\Rej{-}{\varepsilon} = \Rad{(-)}$.
\end{ex}

The statements below are immediate consequences of the definition of preradical.

\begin{lem}
\label{lem:preradicalprops}
Let $\tau$ be a preradical in $\Mod{B}$. Let $N$ and $M$ be $B$-modules, with $N \subseteq M$, and let $(M_i)_{i \in I}$ be a finite family of $B$-modules. The following hold:
\begin{enumerate}
\item $\tau\B{N}\subseteq N\cap \tau\B{M}$;
\item $\B{\tau\B{M}+N}/N \subseteq \tau\B{M/N}$;
\item $\tau\B{\bigoplus_{i \in I} M_i}=\bigoplus_{i \in I} \tau\B{M_i}$.
\end{enumerate}
\end{lem}

\subsubsection{Hereditary and cohereditary preradicals}
\label{subsec:herecohere}
We shall now look at preradicals which satisfy specific properties.

\begin{defi}
\label{defi:idempotent}
A preradical $\tau$ is called \emph{idempotent} if $\tau \circ \tau=\tau$. Symmetrically, we say that $\tau$ is a \emph{radical} if $\tau \circ \B{1/ \tau} =0$. 
\end{defi}
Note that $\Tr{\Theta}{-}$ is an idempotent preradical for every class of $B$-modules $\Theta$. Similarly, the functor $\Rej{-}{\Theta}$ is a radical.

\begin{defi}
\label{defi:hereditary}
A preradical $\tau$ is \emph{hereditary} if $\tau\B{N}=N \cap \tau\B{M}$, for all $M$ and $N$ in $\Mod{B}$ such that $N\subseteq M$. Dually, a preradical $\tau$ is said to be \emph{cohereditary} if $\B{\tau\B{M}+N}/N = \tau\B{M/N}$ for $N\subseteq M$, $M$ and $N$ in $\Mod{B}$.
\end{defi}
\begin{ex}
\label{ex:socrad}
The functors $\Soc{\B{-}}$ and $\Rad{\B{-}}$ are the typical examples of a hereditary preradical and of a cohereditary preradical, respectively.
\end{ex}
\begin{lem}[{\cite[Chapter VI, §1]{MR0389953}}]
\label{lem:hereditary}
Let $\tau$ be a preradical in $\Mod{B}$. The following statements are equivalent:
\begin{enumerate}
\item $\tau$ is hereditary;
\item $\tau$ is a left exact functor;
\item the functor $1/\tau$ preserves monics.
\end{enumerate}
Moreover, any hereditary preradical is idempotent.
\end{lem}
\begin{rem}
\label{rem:dual-hereditary-cohereditary}
There is a result `dual' to Lemma \ref{lem:hereditary} for cohereditary preradicals.
\end{rem}

It is possible to construct hereditary (and cohereditary) preradicals out of special classes of modules.
\begin{defi}
\label{defi:hereditaryclass}
A class $\blacktriangle$ of modules in $\Mod{B}$ is \emph{hereditary} if every submodule of a module in $\Add{\blacktriangle}$ is generated by $\blacktriangle$. 
%Dually, a class $\blacktriangledown$ of modules in $\Mod{B}$ is \emph{cohereditary} if every factor module of a module in $\Add{\blacktriangledown}$ is cogenerated by $\blacktriangledown$.
\end{defi}
%See \cite[Section $1.3$]{thesis} for the proof of the next result.
\begin{lem}
\label{lem:hereditaryclass}
If $\blacktriangle$ is a hereditary class of modules then $\Tr{\blacktriangle}{-}$ is a hereditary preradical in $\Mod{B}$.
\end{lem}
\begin{proof}
Consider $N\subseteq M$, with $M$ and $N$ in $\Mod{B}$. The module $\Tr{\blacktriangle}{M}$ is generated by some module $M'$ which is a (finite) direct sum of modules in $\blacktriangle$. Consider the pullback square
\[
\begin{tikzcd}[ampersand replacement=\&]
M'' \arrow[two heads, dashed]{r} \arrow[hook, dashed]{d} \& N \cap \Tr{\blacktriangle}{M} \arrow[hook]{d} \\
M' \arrow[two heads]{r} \& \Tr{\blacktriangle}{M}
\end{tikzcd}.
\]
As $\blacktriangle$ is a hereditary class, $M''$ is generated by $\blacktriangle$. Hence $N \cap \Tr{\blacktriangle}{M}$ is generated by $\blacktriangle$ as well. Since $\Tr{\blacktriangle}{N}$ is the largest submodule of $N$ generated by $\blacktriangle$, we must have $N \cap \Tr{\blacktriangle}{M} \subseteq \Tr{\blacktriangle}{N}$.
\end{proof}

\subsection{Quasihereditary algebras, RUSQ algebras and the ADR algebra}
We now introduce some notation and state basic results about RUSQ algebras and ADR algebras.
\subsubsection{Quasihereditary algebras}
Given an Artin algebra $B$, we may label the isomorphism classes of simple $B$-modules by the elements of a finite poset $(\Phi , \sqsubseteq )$. Denote the simple $B$-modules by $L_i$, $i \in \Phi$, and use the notation $P_i$ (resp.~$Q_i$) for the projective cover (resp.~injective hull) of $L_i$. 

Let $\Delta\B{i}$ be the \emph{standard module} with label $i\in\Phi$, that is
\[\Delta\B{i}= P_i/\Tr{\bigoplus_{j:\, j\not\sqsubseteq i} P_j}{P_i}.\]
The module $\Delta\B{i}$ is the largest quotient of $P_i$ whose composition factors are all of the form $L_j$, with $j\sqsubseteq i$. Dually, denote the \emph{costandard modules} by $\nabla\B{i}$. The module 
\[\nabla\B{i}=\Rej{Q_i}{\bigoplus_{j:\, j\not\sqsubseteq i} Q_j}\]
is the largest submodule of $Q_i$ with all composition factors of the form $L_j$, with $j \sqsubseteq i$. 

Let $\mathcal{F}\B{\Delta}$ be the category of all $B$-modules which have a $\Delta$-filtration, that is, a filtration whose factors are standard modules. The category $\mathcal{F}\B{\nabla}$ is defined dually. We call $M \in \mathcal{F}(\Delta)$ a \emph{$\Delta$-good module}.

The notation $[M:L]$ will be used for the multiplicity of a simple module $L$ in the composition series of $M$. In a similar manner, $(M:\Delta (i))$ shall denote the multiplicity of $\Delta (i)$ in a $\Delta$-filtration of a module $M$ in $\mathcal{F}(\Delta)$ (this is well defined). %Define $(M:\nabla (i))$, $M \in \mathcal{F}(\nabla)$, in the same way.
\begin{defi}
We say that $\B{B,\Phi, \sqsubseteq}$ is \emph{quasihereditary} if the following hold for every $i \in \Phi$:
\begin{enumerate}
\item $[\Delta (i): L_i]=1$;
\item $P_i \in \mathcal{F}\B{\Delta}$;
\item $\B{P_i: \Delta\B{i}}=1$, and $\B{P_i: \Delta\B{j}}\neq 0 \Rightarrow j\sqsupseteq i$.
\end{enumerate} 
\end{defi}

Let $\B{B,\Phi, \sqsubseteq}$ be quasihereditary. It was proved by Ringel in \cite{last} (see also \cite{donkin}) that there is a unique indecomposable $B$-module $T\B{i}$ (up to isomorphism) for every $i$ which has both a $\Delta$- and a $\nabla$-filtration, with one composition factor labelled by $i$, and all the other composition factors labelled by $j$, $j\sqsubset i$. The standard module $\Delta(i)$ can be embedded in $T(i)$ -- the corresponding monomorphism is a left minimal $\mathcal{F}(\nabla)$-approximation of $\Delta (i)$ and $T(i)/ \Delta (i)$ lies in $\mathcal{F}\B{\Delta}$.
\subsubsection{Strongly and ultra strongly quasihereditary algebras}
Following Ringel (\cite{RingelIyama}), a quasihereditary algebra $\B{B, \Phi, \sqsubseteq}$ is said to be \emph{right strongly quasihereditary} if $\Rad{\Delta\B{i}} \in \mathcal{F}\B{\Delta}$ for all $i \in \Phi$. This property holds if and only if the category $\mathcal{F}\B{\Delta}$ is closed under submodules (see \cite{yey}, \cite[Lemma $4.1$*]{MR1211481} and \cite[Appendix]{RingelIyama}). 

Let $(B, \Phi , \sqsubseteq )$ be an arbitrary quasihereditary algebra, as before. Additionally, suppose that $B$ satisfies the following two conditions:
\begin{description}
\item[(A1)\label{item:A1n}] $\Rad{\Delta\B{i}} \in \mathcal{F}\B{\Delta}$ for all $i \in \Phi$ (i.e.~$B$ is right strongly quasihereditary);
\item[(A2)\label{item:A2n}] $Q_i \in \mathcal{F}\B{\Delta}$ for all $i \in \Phi$ such that $\Rad{\Delta\B{i}}=0$.
\end{description}
We call these algebras \emph{right ultra strongly quasihereditary} algebras (RUSQ algebras, for short). 
\begin{rem}
It was proved in \cite[§$2.5.1$]{thesis} that the definition of RUSQ algebra given in \cite{MR3510398} is equivalent to the one above.
\end{rem}

Let $\B{B,\Phi, \sqsubseteq }$ be a RUSQ algebra. It is always possible to label the elements in $\Phi$ as
\begin{equation}
\label{eq:phiforultrastronglyqhalg}
\Phi = \{(i,j):\, 1 \leq i \leq n ,\, 1 \leq j \leq l_i \},
\end{equation}
for certain $n,l_i \in \Zp$, so that $[\Delta\B{k,l}:L_{i,j}]\neq 0$ implies that $k=i$ and $j\geq l$ (see \cite[§5]{MR3510398}). We shall always assume that the elements in $\Phi$ are labelled in such a way.

The following proposition summarises some properties of the RUSQ algebras.

\begin{prop}[{\cite[§5]{MR3510398}}]
\label{prop:newprop}
Let $\B{B,\Phi, \sqsubseteq }$ be a RUSQ algebra. The following hold:
\begin{enumerate}
\item $\mathcal{F}\B{\Delta}$ is closed under submodules;
\item $\Rad{\Delta\B{i,j}}=\Delta \B{i,j+1}$ for $j < l_i$, and $\Delta\B{i,l_i}=L_{i,l_i}$;
\item each $\Delta\B{i,j}$ is uniserial and has composition factors $L_{i,j}, \ldots, L_{i,l_i}$, ordered from the top to the socle;
\item $Q_{i,l_i} = T\B{i,1}$;
\item for $M\in \mathcal{F}\B{\Delta}$, the number of standard modules appearing in a $\Delta$-filtration of $M$ is given by $\sum_{i=1}^n [M:L_{i,l_i}]$;
\item a module $M$ belongs to $\mathcal{F}\B{\Delta}$ if and only if $\Soc{M}$ is a (finite) direct sum of modules of type $L_{i,l_i}$.
\end{enumerate}
\end{prop}

\subsubsection{The ADR algebra}
Fix an Artin algebra $A$. Given a module $M$ in $\Mod{A}$, we shall denote its \emph{Loewy length} by $\LL{M}$. Let $A$ have Loewy length $L$ (as a left module). We want to study the basic version of the endomorphism algebra of $\bigoplus_{j=1}^L A/ \RAD{j}{A}$.

For this, let $\{P_1, \ldots, P_n\}$ be a complete set of projective indecomposable $A$-modules and let $l_i$ be the Loewy length of $P_i$. Define %\label{eq:generatoradr}
\[
G=G_A:=\bigoplus_{i=1}^{n} \bigoplus_{j=1}^{l_i} P_i/ \RAD{j}{P_i}.
\]

The \emph{Auslander--Dlab--Ringel algebra of $A$} (ADR algebra of $A$) is defined as
\[
R=R_A:=\End{A}{G}\op .
\]
The projective indecomposable $R$-modules are given by
\[
P_{i,j}:=\Hom{A}{G}{P_i/ \RAD{j}{P_i}},
\]
for $1 \leq i \leq n$, $1 \leq j \leq l_i$. %Let $\xi \in R$ be the idempotent corresponding to $\bigoplus_{i=1}^n P_{i,l_i}$ of $R$. Notice that $\xi R \xi$ is a basic algebra of $A$.

Denote the simple quotient of $P_{i,j}$ by $L_{i,j}$ and define
\begin{equation}
\label{eq:posetadr}
\Lambda := \{ (i,j):\, 1 \leq i \leq n, \, 1 \leq j \leq l_i \},
\end{equation}
so that $\Lambda$ labels the simple $R$-modules. Define a partial order, $\unlhd$, on $\Lambda$ by
\begin{equation}
\label{eq:guentanaborrabo}
(i,j) \lhd (k,l) \Leftrightarrow j> l.
\end{equation}
According to \cite[§4]{MR3510398}, $(R, \Lambda, \unlhd)$ is a RUSQ algebra and the labelling $\Lambda$ is compatible with \eqref{eq:phiforultrastronglyqhalg}. 
\begin{thm}[{\cite[§3, §4]{MR3510398}}]
\label{prop:standard}
The ADR algebra $(R, \Lambda, \unlhd)$ is a RUSQ algebra. For every $(i,j)$ in $\Lambda$, we have $\Delta\B{i,j}\cong \RAD{j-1}{P_{i,1}}$ and there are short exact sequences
\begin{equation}
\label{eq:elreidsebastiao}
\begin{tikzcd}[ampersand replacement=\&]
0 \arrow{r} \& \Hom{A}{G}{\Rad{P_i}/ \RAD{j}{P_i}} \arrow{r} \& P_{i,j} \arrow{r} \& \Delta\B{i,j} \arrow{r} \& 0
\end{tikzcd}.
\end{equation}
\end{thm}

\section{\texorpdfstring{$\Delta$}{[Delta]}-semisimple modules and \texorpdfstring{$\Delta$}{[Delta]}-semisimple filtrations}
\label{sec:deltass}
For a quasihereditary algebra $\B{B, \Phi, \sqsubseteq}$, we say that a $B$-module is \emph{$\Delta$-semisimple} if it is a direct sum of standard modules. Every module $M$ in $\mathcal{F}\B{\Delta}$ has some submodule $N$ such that:
\begin{description}
\item[(B)\label{item:b}] $N$ is $\Delta$-semisimple and $M/N$ is in $\mathcal{F}\B{\Delta}$.
\end{description}
Given a module $M$ in $\mathcal{F}\B{\Delta}$, we may consider the submodules of $M$ which are maximal with respect to property \ref{item:b}. The module $M$ may have more than one such submodule (see Example $2.20$ in \cite{MR2642018}). However, according to \cite{MR2642018}, the submodules of $M$ which are maximal with respect to \ref{item:b} are unique up to isomorphism.

Suppose now that $B$ is a RUSQ algebra. The $\Delta$-sem\-i\-sim\-ple modules over RUSQ algebras are particularly well behaved. As we will see in Corollary \ref{cor:ultrastronglyqhdeltass}, the property of being $\Delta$-semisimple is closed under submodules in this case. Furthermore, every module $M$ in $\mathcal{F}\B{\Delta}$ has exactly one submodule $D_M$ which is maximal with respect to property \ref{item:b}. The module $D_M$ is actually the unique maximal $\Delta$-semisimple submodule of $M$ (with respect to inclusion). Moreover, $D_M$ will be obtained by applying a certain hereditary preradical (as in Definition \ref{defi:hereditary}) to the module $M$. Since $M/D_M$ still lies in $\mathcal{F}\B{\Delta}$, we may proceed iteratively and define the $\Delta$-semisimple filtration (which will be unique) and the $\Delta$-semisimple length of any module in $\mathcal{F}\B{\Delta}$. 

\subsection{\texorpdfstring{$\Delta$}{[Delta]}-semisimple modules}
\label{subsec:moti}

We are interested in submodules of $\Delta$-good modules which are maximal with respect to property \ref{item:b}. As a consequence of Theorem $2.17$ in \cite{MR2642018}, these are unique up to isomorphism.
\begin{thm}[{\cite[Theorem $2.17$]{MR2642018}}]
Let $\B{B, \Phi , \sqsubseteq}$ be a quasihereditary algebra, and let $M$ be in $\mathcal{F}\B{\Delta}$. Any two submodules of $M$ which are maximal with respect to property \ref{item:b} are isomorphic.
\end{thm}

Note that $\Delta$-semisimple modules are in general well behaved with respect to quotients in the following way: every $\Delta$-good factor module of a $\Delta$-semisimple module is still $\Delta$-semisimple. This assertion follows from the fact that $\mathcal{F}\B{\Delta}$ is closed under taking kernels of epimorphisms (\cite[Lemma $1.5$]{MR1211481}) and from Theorem $3.2$ in \cite{MR2642018}.

We wish to study the $\Delta$-semisimple modules over a RUSQ algebra $\B{B, \Phi, \sqsubseteq}$. We shall assume that the set $\Phi$ is as described in \eqref{eq:phiforultrastronglyqhalg}. In this subsection we prove some key properties of the $\Delta$-semisimple modules over RUSQ algebras. Namely, we show that the property of being $\Delta$-semisimple is closed under taking submodules.

%According to Proposition \ref{prop:newprop}, the standard modules over a RUSQ algebra are uniserial and satisfy $\Rad{\Delta\B{i,j}}=\Delta\B{i,j+1}$. The module $\Delta\B{i,j}$ has composition factors $L_{i,j}, \ldots, L_{i,l_i}$, ordered from the top to the socle.
\begin{lem}
\label{lem:ultrastronglyses}
Let $\B{B, \Phi, \sqsubseteq}$ be a RUSQ algebra. Let $M$ be in $\Mod{B}$ and suppose that there is a short exact sequence
\begin{equation}
\label{eq:christophlisboa}
\begin{tikzcd}[ampersand replacement=\&]
0 \arrow{r} \& \Delta\B{k,l} \arrow{r}{f} \& M  \arrow{r}\& \Delta \B{i,j} \arrow{r} \& 0
\end{tikzcd},
\end{equation}
with $(k,l), (i,j) \in \Phi$. If $\Soc{M}\not\cong \Soc{\Delta\B{k,l}}$, then \eqref{eq:christophlisboa} splits.
\end{lem}
\begin{proof}
Note that $\Soc{\B{\Ima{f}}}= \Soc{\Delta\B{k,l}}= L_{k,l_k}$ (see Proposition \ref{prop:newprop}, part 3). As $\Soc{M}\not\cong \Soc{\Delta\B{k,l}}$, there is some nonzero submodule $M'$ of $M$ such that $\Ima{f} \cap M'=0$. Let $g: \Ima{f} \oplus M' \longrightarrow Q_{k,l_k}$ be the morphism which embeds $\Ima{f}$ in $Q_{k,l_k}$ and maps $M'$ to zero. By the injectivity of $Q_{k,l_k}$, $g$ extends to a map $g': M \longrightarrow Q_{k,l_k}$. Note that $\Ima{f}\cap \Ker{g'} =0$ as $g'|_{\Ima{f}}=g|_{\Ima{f}}$ is an injective map. Thus $\Ima{f}\oplus\Ker{g'}$ is a submodule of $M$.

By part 4 of Proposition \ref{prop:newprop}, $Q_{k,l_k}$ is in $\mathcal{F}(\Delta)$. Since $\mathcal{F}\B{\Delta}$ is closed under taking submodules, then both $\Ima{g'}$ and $\Ker{g'}$ lie in $\mathcal{F}\B{\Delta}$. Denote by $\Delta |N|$ the number of standard modules appearing in a $\Delta$-filtration of $N\in \mathcal{F}\B{\Delta}$. As $\Delta |M|=2$, $\Delta |\Ima{g'}|>0$ and $\Delta |\Ker{g'}| > 0$, it follows that $\Delta |\Ima{g'}|=\Delta |\Ker{g'}|=1$. So either $\Ima{g'}\cong \Delta(k,l)$ and $\Ker{g'}\cong \Delta(i,j)$, or $\Ima{g'}\cong \Delta(i,j)$ and $\Ker{g'}\cong \Delta(k,l)$.

If $\Ker{g'}\cong \Delta(i,j)$, then the submodule $\Ima{f} \oplus \Ker{g'}$ of $M$ must coincide with $M$ (as both modules have the same Jordan--H\"{o}lder length). In this case the monic $f$ splits. 

If $\Ima{g'}\cong \Delta(i,j)$, then $(k,l) \sqsubseteq (i,j) $ as
\[\Delta\B{k,l}\cong \Ima{f} \cong \Ima{g} \subseteq \Ima{g'}.\]
But then part $(b)$ of Lemma $1.3$ in \cite{MR1211481} implies that \eqref{eq:christophlisboa} is a split exact sequence.
\end{proof}

We now use the previous result to give a characterisation of the $\Delta$-semisimple modules over a RUSQ algebra.
\begin{cor}
\label{cor:ultrastronglyqhdeltass}
Let $\B{B, \Phi, \sqsubseteq}$ be a RUSQ algebra and let $M$ be in $\mathcal{F}\B{\Delta}$. Then $M$ is $\Delta$-semisimple if and only if the number of simple summands of $\Soc{M}$ coincides with the number of factors in a $\Delta$-filtration of $M$. Moreover, any submodule of a $\Delta$-semisimple module is still $\Delta$-semisimple.
\end{cor}
\begin{proof}
Let $M$ be in $\mathcal{F}\B{\Delta}$. Denote by $\mathcal{P}(M)$ the following assertion: ``the number of simple summands of $\Soc{M}$ coincides with the number of factors in a $\Delta$-filtration of $M$". By parts 5 and 6 of Proposition \ref{prop:newprop}, $\mathcal{P}(M)$ is true if and only if the composition factors of $M$ of type $L_{x,l_x}$ are exactly the summands of its socle. From this equivalence, it is easy to see that the truth of $\mathcal{P}(M)$ implies the truth of $\mathcal{P}(N)$ for $N\subseteq M$. Let $N \in \mathcal{F}(\Delta)$ be a submodule of $M$ such that $M/N \in \mathcal{F}\B{\Delta}$. Using Proposition \ref{prop:newprop}, we also conclude that the truth of $\mathcal{P}(M)$ implies the truth of $\mathcal{P}(M/N)$.

If $M$ is a $\Delta$-semisimple module then $\mathcal{P}(M)$ is clearly true. Suppose now that $\mathcal{P}(M)$ holds for $M\in \mathcal{F}\B{\Delta}$. We wish to show that $M$ is $\Delta$-semisimple. We prove this by induction on the number $z$ of factors in a $\Delta$-filtration of $M$. If $z=1$ the result is obvious. Suppose now that $z \geq 2$. Let $M_1$ be a submodule $M$ satisfying $M/M_1\in\mathcal{F}\B{\Delta}$ and $M_1\cong \Delta (i,j)$ for some $(i,j) \in \Phi$. By the remark in the first paragraph, $\mathcal{P}(M/M_1)$ holds. Using induction, we conclude that $M/M_1$ is $\Delta$-semisimple. Therefore $M/M_1=\bigoplus_{i=1}^{z-1} N_i/M_1$, where each $N_i/M_1$ is isomorphic to a standard module. Applying again the observations in the first paragraph, we deduce that assertion $\mathcal{P}(N_i)$ must hold, so, by induction, each $N_i$ is a $\Delta$-semisimple module. Using that both $M_1$ and $N_i/M_1 $ are standard modules, we conclude that $N_i= M_1 \oplus M_{i+1}$, where $M_{i+1}$ is a submodule of $M$ isomorphic to a standard module. Note that $\bigoplus_{i=1}^z M_i$ is a submodule of $M$ which has the same Jordan--H\"{o}lder length as $M$. This implies that $M=\bigoplus_{i=1}^z M_i$, which proves that $M$ is $\Delta$-semisimple. We have just shown that $M$ is a $\Delta$-semisimple module if and only if assertion $\mathcal{P}(M)$ holds.

Let now $N$ be a submodule of a $\Delta$-semisimple module $M$. Then $\mathcal{P}(M)$ is true, which implies that $\mathcal{P}(N)$ holds. Therefore $N$ is a $\Delta$-semisimple module.
\end{proof}
In the next subsection we are going to show that every $\Delta$-good module over a RUSQ algebra has a unique maximal $\Delta$-semi\-sim\-ple submodule. First, we check that arbitrary quasihereditary algebras do not possess this property.
\begin{ex}
\label{ex:karinparker}
Consider the quiver
\[
Q=
\begin{tikzcd}[ampersand replacement=\&]
0 \& 1 \arrow[bend left]{l}{\alpha} \arrow[bend right]{l}[swap]{\varepsilon} \\
3 \arrow{u}{\delta}\arrow[bend left]{r}{\gamma_1}\& 2 \arrow{u}{\beta} \arrow[bend left]{l}{\gamma_0}
\end{tikzcd},
\]
and the bound quiver algebra $B=KQ/ I$, where $I$ is the ideal generated by the elements $\varepsilon \beta - \delta \gamma_0$ and $\gamma_0 \gamma_1$. The algebra $B$ is quasihereditary with respect to the labelling poset $0 < 1 < 2 < 3$. The modules
\[
\begin{tikzcd}[ampersand replacement=\&, row sep =tiny, column sep = tiny]
0
\end{tikzcd}, \, \,
\begin{tikzcd}[ampersand replacement=\&, row sep =tiny, column sep = tiny]
\& 1 \arrow[dash]{dl}[swap]{\varepsilon} \arrow[dash]{dr}{\alpha} \& \\
0 \& \& 0
\end{tikzcd} , \, \,
\begin{tikzcd}[ampersand replacement=\&, row sep =tiny, column sep = tiny]
2 \arrow[dash]{d}
\\ 1 \arrow[dash]{d}{\alpha} \\
0
\end{tikzcd}, \,\,
\begin{tikzcd}[ampersand replacement=\&, row sep =tiny, column sep = tiny]
\& 3 \arrow[dash]{dl} \arrow[dash]{dr} \& \\
0 \& \& 2 \arrow[dash]{d} \\
 \& \& 1\arrow[dash]{d}{\alpha} \\
 \&\& 0
\end{tikzcd}.
\]
are the corresponding standard $B$-modules. The projective cover $P_2$ of the simple module with label 2 has the following structure
\[
\begin{tikzcd}[ampersand replacement=\&, row sep =tiny, column sep = tiny]
\& \& 2 \arrow[dash]{dl} \arrow[dash]{dr} \& \&\\
\& 3 \arrow[dash]{dl} \arrow[dash]{dr} \& \& 1 \arrow[dash]{dl}[swap]{\varepsilon} \arrow[dash]{dr}{\alpha}\& \\
2\arrow[dash]{d} \& \& 0\& \& 0  \\
1 \arrow[dash]{d}{\alpha} \&\&\&\& \\ 
0 \&\&\&\& 
\end{tikzcd}.\]
The modules $\Delta\B{1} \oplus \Delta\B{2}$ and $\Delta\B{3} \oplus \Delta\B{0}$ are both maximal $\Delta$-semisimple submodules of $P_2$. The quotient of $P_2$ by each of these submodules does not belong to $\mathcal{F}\B{\Delta}$, i.e.~none of these submodules of $P_2$ satisfies property \ref{item:b}.
\end{ex}

\subsection{The preradical \texorpdfstring{$\delta$}{[delta]} and \texorpdfstring{$\Delta$}{[Delta]}-semisimple filtrations}
\label{subsec:thepreradicaldelta}

Let $\B{B ,\Phi, \sqsubseteq}$ be an arbitrary quasihereditary algebra. As pointed out in the previous subsection, the submodules of a module $M$ in $\mathcal{F}\B{\Delta}$ which are maximal with respect to property \ref{item:b} are all isomorphic, but they are not necessarily unique. We have also seen that a module $M$ in $\mathcal{F}\B{\Delta}$ may have more than one maximal $\Delta$-semisimple submodule with respect to inclusion (Example \ref{ex:karinparker}). We shall prove that both these maximal submodules are unique and actually coincide when the underlying algebra is a RUSQ algebra. For this, we use the general theory of preradicals introduced in Subsection \ref{sec:preradicals}.
%Recall the definition of a hereditary class (Definition \ref{defi:hereditaryclass}).
\begin{lem}
\label{lem:Deltahereditary}
Let $(B,\Phi, \sqsubseteq)$ be a RUSQ algebra. The corresponding set $\Delta$ of standard $B$-modules is a hereditary class in $\Mod{B}$. In particular, $\Tr{\Delta}{-}$ is a hereditary preradical in $\Mod{B}$.
\end{lem}
\begin{proof}
Let $N$ be a submodule of a module in $\Add{\Delta}$, so $N$ is contained in some $\Delta$-semisimple module $M$. By Corollary \ref{cor:ultrastronglyqhdeltass}, $N$ is still $\Delta$-semisimple, so it is trivially generated by $\Delta$. Hence the set $\Delta$ is hereditary. Lemma \ref{lem:hereditaryclass} implies that $\Tr{\Delta}{-}$ is a hereditary preradical in $\Mod{B}$.
\end{proof}
\begin{rem}
The preradical $\Tr{\Delta}{-}$ is not usually hereditary for an arbitrary quasihereditary algebra $B$ (not even if $B$ is right strongly quasihereditary).
\end{rem}

From now onwards we shall denote the functor $\Tr{\Delta}{-}$ by $\delta$. 
\begin{defi}
\label{defi:preradicaldelta}
For a RUSQ algebra $(B,\Phi, \sqsubseteq)$, let $\delta$ be the hereditary preradical $\Tr{\Delta}{-}$ in $\Mod{B}$.
\end{defi}
Next, we give a description of the submodule $\delta\B{M}$ of a module $M\in\mathcal{F}\B{\Delta}$.
\begin{prop}
\label{prop:deltahereditary2}
Let $(B,\Phi, \sqsubseteq)$ be a RUSQ algebra, and let $M \in \mathcal{F}\B{\Delta}$. Then $\delta\B{M}$ is the largest $\Delta$-semisimple submodule of $M$. Furthermore, $M/ \delta\B{M}$ lies in $\mathcal{F}\B{\Delta}$. In particular, $\delta\B{M}$ is the largest submodule of $M$ satisfying property \ref{item:b}.
\end{prop}
\begin{proof}
By the definition of $\Tr{\Delta}{-}$, there is an epic $f$ from a $\Delta$-semisimple module $M'$ to $\delta\B{M}$. Note that both $\delta\B{M}$ and $\Ker{f}$ are in $\mathcal{F}\B{\Delta}$, since this category is closed under submodules. As a consequence, $f$ must be a split epic. Hence $\delta\B{M}$ is $\Delta$-semisimple. By the definition of $\Tr{\Delta}{-}$ it is clear that every $\Delta$-semisimple submodule of $M$ must be contained in $\delta\B{M}$. This shows that $\delta\B{M}$ is the largest $\Delta$-semisimple submodule of $M \in \mathcal{F}\B{\Delta}$.

To conclude this proof it is enough to show that $M/\delta\B{M}$ lies in $\mathcal{F}\B{\Delta}$. We start by proving that this holds for the injective modules $Q_{i,l_i}=T\B{i,1}$ (recall Proposition \ref{prop:newprop}). Note that $\Delta\B{i,1} \subseteq \delta\B{Q_{i,l_i}}$, as $\Delta\B{i,1}$ is a submodule of $T\B{i,1}$. Since $Q_{i,l_i}$ has simple socle $L_{i,l_i}$, then $\delta\B{Q_{i,l_i}}$ has to be isomorphic to some standard module $\Delta\B{i,j}$. But then we must have $\Delta\B{i,1}=\delta\B{Q_{i,l_i}}$, and consequently $Q_{i,l_i}/ \delta\B{Q_{i,l_i}} = T\B{i,1}/\Delta\B{i,1}$ is in $\mathcal{F}\B{\Delta}$. Let now $Q$ be a finite direct sum of injective modules of type $Q_{i,l_i}$. The module $Q/\delta\B{Q}$ still lies in $\mathcal{F}\B{\Delta}$ because preradicals preserve finite direct sums (see part 3 of Lemma \ref{lem:preradicalprops}). Consider now $M$ in $\mathcal{F}\B{\Delta}$. By Proposition \ref{prop:newprop}, the injective hull $q_0: M \longrightarrow Q_0\B{M}$ of $M\in \mathcal{F}\B{\Delta}$ is such that $Q_0\B{M}$ is a direct sum of injectives of type $Q_{i,l_i}$. By part 3 of Lemma \ref{lem:hereditary}, $q_0$ gives rise to a monic $M/\delta\B{M} \longrightarrow Q_0\B{M}/ \delta(Q_0\B{M})$, and by our previous observation $Q_0\B{M}/ \delta(Q_0\B{M})$ lies in $\mathcal{F}\B{\Delta}$. As $\mathcal{F}\B{\Delta}$ is closed under submodules, the module $M/ \delta\B{M}$ belongs to $\mathcal{F}\B{\Delta}$.
\end{proof}

\begin{ex}
Note that for an arbitrary quasihereditary algebra the modules $\delta\B{M}$, $M \in \mathcal{F}\B{\Delta}$, are not usually $\Delta$-semisimple (not even $\Delta$-good). Indeed, for the algebra in Example \ref{ex:karinparker}, we have $\delta\B{P_2}=\Tr{\Delta}{P_2}=\Rad{P_2}$, which is not $\Delta$-semisimple.
\end{ex}
\subsubsection{Filtrations arising from preradicals}
\label{subsec:filpreradicals}
Our next goal is to define $\Delta$-semisimple filtration and $\Delta$-semisimple length for modules in $\mathcal{F}\B{\Delta}$, over some RUSQ algebra $B$. For this, some elementary results about preradicals are needed.

Let $\tau$ and $\upsilon$ be preradicals (over an arbitrary Artin algebra $B$). Write $\tau \leq \upsilon$ if $\tau$ is a subfunctor of $\upsilon$. The functor $\tau \circ \upsilon$ is a preradical, and $\tau \circ \upsilon \leq \upsilon$. For $M$ in $\Mod{B}$ define $\tau \bullet \upsilon \B{M}$ as the submodule of $M$ containing $\upsilon\B{M}$, satisfying
\[
\tau \B{M / \upsilon \B{M}}= \tau \bullet \upsilon \B{M}/ \upsilon\B{M}.
\]
The operator $\tau \bullet \upsilon$ is still a predadical. By construction, $\upsilon \leq \tau \bullet \upsilon$. 

By the characterisation of hereditary radicals given in Lemma \ref{lem:hereditary}, it follows that $\tau \circ \upsilon$ is hereditary if both $\tau$ and $\upsilon$ are hereditary. We also have that $\tau \bullet \upsilon$ is hereditary, whenever $\tau$ and $\upsilon$ are both hereditary -- the functor $1/ \B{\tau \bullet \upsilon}$ is naturally isomorphic to $\B{1/ \tau} \circ \B{1/ \upsilon}$.

Similarly to the composition of preradicals, the operation $\bullet$ is associative. Given a preradical $\tau$, let $\tau^0$ be the identity functor in $\Mod{B}$ and let $\tau_0$ be the zero preradical. For $m \in \Zp$, define $\tau^m:= \tau \circ \tau^{m-1}$ and $\tau_m:= \tau \bullet \tau_{m-1}$. We summarise the properties of these preradicals.
\begin{lem}
\label{lem:summary}
Let $\tau$ be a preradical in $\Mod{B}$.
\begin{enumerate}
\item For every $m \geq 1$, $\tau^m \leq \tau^{m-1}$ and $\tau_{m-1} \leq \tau_m$.
%\item Given $m,m' \geq 0$, then $\tau^m \circ \tau^{m'}=\tau^{m+m'}$.
%\item Given $m,m' \geq 0$, then $\tau_m \bullet \tau_{m'}=\tau_{m+m'}$.
\item For every $M$ in $\Mod{B}$ there is $m \geq 0$ such that $\tau^{m} \B{M}= \tau^{m+1}\B{M}$.
\item For every $M$ in $\Mod{B}$ there is $m \geq 0$ such that $\tau_{m} \B{M}= \tau_{m+1}\B{M}$.
%\item If $\tau$ is a radical then $\tau_m=\tau$, for every $m \geq 1$.
%\item If $\tau$ is idempotent then $\tau^m=\tau$, for every $m \geq 1$.
\end{enumerate}
\end{lem}

The preradicals $\tau_m$ (and $\tau^m$), $m \in \Znn$, give rise to special filtrations. %The proof of the next two lemmas is straightforward and may be found in \cite[Section $1.3$]{thesis}. %We are specially interested in filtrations arising from the latter family of preradicals.
\begin{lem}[{\cite[§$1.3.3$]{thesis}}]
\label{lem:lemita1}
Let $\tau$ be a preradical. Suppose that $\tau\B{M}\neq 0$ for every nonzero $B$-module $M$. Given $M$ in $\Mod{B}$, there is a unique integer $l^{\B{\tau , \bullet}}\B{M} = n \geq 0$ such that $\tau_{n}\B{M}=M$, and $ \tau_{m-1}\B{M} \subset \tau_m \B{M}$ for every $m$ satisfying $1 \leq m \leq n$. Moreover, for $m \leq l^{\B{\tau , \bullet}}\B{M}$, we have
\[
l^{\B{\tau , \bullet}}\B{M/ \tau_m \B{M}} = n -m.
\]
\end{lem}

\begin{lem}[{\cite[§$1.3.3$]{thesis}}]
\label{lem:lemita2}
Let $\tau$ be a hereditary preradical. Then $\tau_m$ is also a hereditary preradical, and $\tau_m \circ \tau_{m'}=\tau_{\min\{m,m'\}}$ for every $m, m' \geq 0$. Furthermore, if $\tau\B{M}\neq 0$ for every $M\neq 0$, the following hold for $N$ and $M$ in $\Mod{B}$:
\begin{enumerate}
\item if $N \subseteq M$ then $l^{\B{\tau , \bullet}}\B{N} \leq l^{\B{\tau , \bullet}}\B{M}$;
%\item if $N \subseteq \tau_m\B{M}$, then $l^{\B{\tau , \bullet}}\B{N} \leq m$;
\item if $m \leq l^{\B{\tau , \bullet}}\B{M}$, then $\tau_m\B{M}$ is the largest submodule $N$ of $M$ such that $l^{\B{\tau , \bullet}}\B{N}=m$.
\end{enumerate}
\end{lem}

\subsubsection{\texorpdfstring{$\Delta$}{[Delta]}-semisimple filtrations and \texorpdfstring{$\Delta$}{[Delta]}-semisimple length}

Suppose once again that $(B, \Phi, \sqsubseteq)$ is a RUSQ algebra, and consider the hereditary preradical $\delta$. Note that $\delta\B{M}\neq 0$ for every nonzero module $M$ in $\Mod{B}$ as
\[ \Soc{M} \subseteq \Tr{\Delta}{M} =\delta\B{M}. \]
In fact, we have $\Soc{M}=\Soc{\delta\B{M}}$. We may construct the preradicals $\delta_m$ in $\Mod{B}$ defined recursively in §\ref{subsec:filpreradicals}. Then Lemmas \ref{lem:summary}, \ref{lem:lemita1} and \ref{lem:lemita2} hold for the preradicals $\delta_{m}$. In particular, $\delta_{m}$ is a hereditary preradical for every $m \in \Znn$.

\begin{lem}
\label{lem:deltasemisimplefiltration}
Let $(B,\Phi, \sqsubseteq)$ be a RUSQ algebra. If $M$ is in $\mathcal{F}\B{\Delta}$ then so is $M/ \delta_m \B{M}$, for any $m \geq 0$.
\end{lem}
\begin{proof}
By Proposition \ref{prop:deltahereditary2}, the claim holds for $m=1$. Suppose $m \geq 2$. Then
\begin{align*}
M/ \delta_m \B{M} & \cong \B{M / \delta\B{M}} / \B{\delta_{m-1} \bullet \delta \B{M} / \delta\B{M} }  \\ &= \B{M/ \delta\B{M}} / \B{\delta_{m-1}\B{M / \delta\B{M}}},
\end{align*}
so by induction $M / \delta_m \B{M}$ belongs to $\mathcal{F}\B{\Delta}$.
\end{proof}
Given a module $M$ in $\mathcal{F}\B{\Delta}$, we may consider the filtration
\begin{equation}
\label{eq:deltasemisimplefiltartion}
0 \subset \delta\B{M} \subset \cdots \subset \delta_{m}\B{M}=M ,
\end{equation}
where $m=l^{\B{\delta, \bullet}}\B{M}$ is as defined in Lemma \ref{lem:lemita1}. The factors of this filtration are $\Delta$-semisimple: by Lemma \ref{lem:deltasemisimplefiltration} and Proposition \ref{prop:deltahereditary2} the modules $ \delta_i (M)/\delta_{i-1} (M)= \delta(M/\delta_{i-1}(M))$ are $\Delta$-semisimple. We call \eqref{eq:deltasemisimplefiltartion} the \emph{$\Delta$-semisimple filtration} of $M \in \mathcal{F}\B{\Delta}$.
\begin{defi}
\label{defi:deltasemisimplelength}
The \emph{$\Delta$-semisimple length} of a module $M$ in $\mathcal{F}\B{\Delta}$, denoted by $\dssl{M}$, is the length of the $\Delta$-semisimple filtration of $M$, i.e.~it is given by the number $l^{\B{\delta, \bullet}}\B{M}$ (as in Lemma \ref{lem:lemita1}).
\end{defi}

\section{\texorpdfstring{$\Delta$}{[Delta]}-semisimple filtrations of modules over the ADR algebra}
\label{sec:manuela}
The ADR algebra of an Artin algebra $A$,  $R=\B{R_A, \Lambda, \unlhd}$, is our prototype of a RUSQ algebra. We now prove some results specific to the $\Delta$-semisimple filtrations of $\Delta$-good modules over the ADR algebra. Throughout this section the underlying quasihereditary algebra will be $\B{R, \Lambda, \unlhd}$, where the poset $\B{\Lambda, \unlhd}$ is as defined in \eqref{eq:posetadr} and \eqref{eq:guentanaborrabo}. For the proof of the next results note that the left exact functor $\Hom{A}{G}{-}$ is fully faithful since $G$ is a generator of $\Mod{A}$ (see \cite[§8--§10]{MR0349747}).
\begin{lem}
\label{lem:P_{i,L}}
Let $M_1$ and $M_2$ be in $\Mod{A}$, with $M_1 \subseteq M_2$. There is a canonical embedding
\[
\begin{tikzcd}[ampersand replacement=\&]
\Hom{A}{G}{M_2}/\Hom{A}{G}{M_1} \arrow[hook]{r}{\iota} \& \Hom{A}{G}{M_2/M_1}
\end{tikzcd}
\]
and the induced morphisms
\begin{gather*}\begin{tikzpicture}[baseline= (a).base]
\node[scale=.9] (a) at (0,0){
\begin{tikzcd}[ampersand replacement=\&, column sep=normal]
\Hom{R}{P_{i,l_i}}{\Hom{A}{G}{M_2}/\Hom{A}{G}{M_1}} \arrow[hook]{r}{\iota_*} \& \Hom{R}{P_{i,l_i}}{\Hom{A}{G}{M_2/M_1}},
\end{tikzcd}
};
\end{tikzpicture}\\
\begin{tikzpicture}[baseline= (a).base]
\node[scale=.9] (a) at (0,0){
\begin{tikzcd}[ampersand replacement=\&, column sep=normal]
\Hom{R}{\Hom{A}{G}{M_2/M_1}}{Q_{i,l_i}}\arrow[two heads]{r}{\iota^*} \&  \Hom{R}{\Hom{A}{G}{M_2}/\Hom{A}{G}{M_1}}{Q_{i,l_i}} 
\end{tikzcd}
};
\end{tikzpicture}
\end{gather*}
are isomorphisms.
\end{lem}
\begin{proof}
The functor $\Hom{A}{G}{-}$ is left exact. Thus, it maps the canonical epic $\pi:M_2 \longrightarrow M_2/ M_1$ to the morphism $\pi_*$, which factors as
\[
\begin{tikzcd}[ampersand replacement=\&, column sep=tiny]
\Hom{A}{G}{M_2}\arrow[two heads]{rd}[swap]{\varpi} \arrow{rr}{\pi_*} \& \& \Hom{A}{G}{M_2/M_1} \\
\& \Hom{A}{G}{M_2}/ \Hom{A}{G}{M_1} \arrow[hook]{ru}[swap]{\iota}\&
\end{tikzcd}.
\]
Consider the monic $\iota_*$ obtained by applying the functor $\Hom{R}{P_{i,l_i}}{-}$ to $\iota$. Let $f_*$ be in $\Hom{R}{P_{i,l_i}}{\Hom{A}{G}{M_2/M_1}}$. Then $f_*=\Hom{A}{G}{f}$, for a map $f:P_i \longrightarrow M_2/ M_1$ in $\Mod{A}$. Since $P_i$ is projective, $f= \pi \circ t$ for some $t: P_i \longrightarrow M_2$. So $f_*=\pi_* \circ t_*= \iota \circ \varpi \circ t_*=\iota_*(\varpi \circ t_*)$, where $t_*=\Hom{A}{G}{t}$. This shows that $\iota_*$ is surjective, hence it is an isomorphism. The proof that $\iota^*$ is an isomorphism is analogous.
\end{proof}

Let $M_1$ and $M_2$ be in $\Mod{A}$, with $M_1 \subseteq M_2$. We shall regard the canonical embedding in Lemma \ref{lem:P_{i,L}},
\[
\begin{tikzcd}[ampersand replacement=\&]
\Hom{A}{G}{M_2}/\Hom{A}{G}{M_1} \arrow[hook]{r}{\iota} \& \Hom{A}{G}{M_2/M_1}
\end{tikzcd},
\]
as an inclusion of $R$-modules. According to Lemma $3.6$ in \cite{MR3510398}, $\Hom{A}{G}{M}$ lies in $\mathcal{F}\B{\Delta}$ for every $M$ in $\Mod{A}$. Since the category $\mathcal{F}\B{\Delta}$ is closed under submodules then both $\Hom{A}{G}{M_2}/\Hom{A}{G}{M_1}$ and $\Hom{A}{G}{M_2/M_1}$ are $\Delta$-good modules. Lemma \ref{lem:P_{i,L}} is hinting at a close relation between the $\Delta$-filtrations of the modules $\Hom{A}{G}{M_2}/\Hom{A}{G}{M_1}$ and $\Hom{A}{G}{M_2/M_1}$. We spell out this idea below.

\begin{cor}
\label{cor:neweasyses}
Let $M_1$ and $M_2$ be in $\Mod{A}$, with $M_1 \subseteq M_2$. Write $M=\Hom{A}{G}{M_2/M_1}$ and $M'=\Hom{A}{G}{M_2}/\Hom{A}{G}{M_1}$. All the composition factors of $M$ of type $L_{i,l_i}$ appear as composition factors of its submodule $M'$. In particular, $M$ and $M'$ have the same number of composition factors of type $L_{i,l_i}$. Moreover, $M$ and $M'$ lie in $\mathcal{F}\B{\Delta}$, $\Soc{M}= \Soc{M'}$, and the modules $M$ and $M'$ are filtered by the same number of standard modules.
\end{cor}
\begin{proof}
By Lemma \ref{lem:P_{i,L}}, all the composition factors of $M$ isomorphic to $L_{i,l_i}$ appear as composition factors of its submodule $M'$. As $M$ lies in $\mathcal{F}\B{\Delta}$ then, by Proposition \ref{prop:newprop}, $\Soc{M}$ is a direct sum of simples of type $L_{i,l_i}$. Thus $\Soc{M'}=\Soc{M}$. Part 5 of Proposition \ref{prop:newprop} implies that the $\Delta$-filtrations of $M$ and $M'$ have the same number of factors.
\end{proof}
As we shall see next, the socle series of an $A$-module $M$ gives rise to the $\Delta$-semisimple filtration of $\Hom{A}{G}{M}$ in $\mathcal{F}\B{\Delta}$.
\begin{lem}
\label{lem:soc0}
Let $M$ be in $\Mod{A}$. Then
\begin{equation}
\label{eq:teresapreguicosa}
\Hom{A}{G}{\SOC{j}{M}}=\Tr{\bigoplus_{(k,l):\, l \leq j}P_{k,l}}{\Hom{A}{G}{M}}.
\end{equation}
Moreover, if $\SOC{j}{M}/ \SOC{j-1}{M} = \bigoplus_{\theta \in \Theta} L_{x_{\theta}}$, then
\[
\Hom{A}{G}{\SOC{j}{M}}/ \Hom{A}{G}{\SOC{j-1}{M}} = \bigoplus_{\theta \in \Theta} \Delta\B{x_{\theta},j}.
\]
\end{lem}
\begin{proof}
By \cite[Proposition $10.2$]{MR0349747} (see also \cite[Lemma $3.3$]{MR3510398}), $\Hom{A}{G}{\SOC{j}{M}}$ is generated by projectives $P_{k,l}$ satisfying $l \leq j$. This proves one of the inclusions in \eqref{eq:teresapreguicosa}. Consider now an arbitrary morphism $f_*: P_{k,l}\longrightarrow \Hom{A}{G}{M}$, with $l\leq j$. Note that $f_*= \Hom{A}{G}{f}$ for a certain map $f: P_{k}/ \RAD{l}{P_k} \longrightarrow M$.
Clearly, $\Ima{f} \subseteq \SOC{j}{M}$. But then
\[
\Ima{f_*} \subseteq \Hom{A}{G}{\Ima{f}} \subseteq \Hom{A}{G}{\SOC{j}{M}}.
\]
As $f_*$ was chosen arbitrarily, the other inclusion follows. This proves identity \eqref{eq:teresapreguicosa}.

To prove the second claim in the statement of the lemma, set
\[
M':=\Hom{A}{G}{\SOC{j}{M}}/ \Hom{A}{G}{\SOC{j-1}{M}},
\]
and assume that $\SOC{j}{M}/ \SOC{j-1}{M}$ is isomorphic to $\bigoplus_{\theta \in \Theta} L_{x_{\theta}}$. Recall that $\Delta (i,1)$ is isomorphic to $\Hom{A}{G}{L_i}$ (see Theorem \ref{prop:standard}). Lemma \ref{lem:P_{i,L}} and Corollary \ref{cor:neweasyses} imply that $M'$ is contained in
\[
\Hom{A}{G}{\SOC{j}{M}/ \SOC{j-1}{M}}=\bigoplus_{\theta \in \Theta} \Hom{A}{G}{L_{x_{\theta}}}= \bigoplus_{\theta \in \Theta} \Delta\B{x_{\theta},1}
\]
and that these modules have the same socle. By Corollary \ref{cor:ultrastronglyqhdeltass}, $M'$ is $\Delta$-semisimple. Finally, by the identity \eqref{eq:teresapreguicosa} (applied to $j$ and $j-1$), the module $M'$ must be generated by projectives of type $P_{i,j}$. This proves the second assertion of the lemma.
\end{proof}

Lemma \ref{lem:soc0} and Theorem \ref{thm:socdelta} are very useful to compute examples. For the proof of the next result, recall the characterisation of the preradical $\delta$ in Subsection \ref{subsec:thepreradicaldelta}, namely Proposition \ref{prop:deltahereditary2} and Lemma \ref{lem:deltasemisimplefiltration}.
\begin{thm}
\label{thm:socdelta}
Let $M$ be in $\Mod{A}$. The socle series of $M$ induces the $\Delta$-semisimple filtration of $\Hom{A}{G}{M}$. Formally,
\[
\delta_m \B{\Hom{A}{G}{M}} = \Hom{A}{G}{\SOC{m}{M}},
\]
for all $m \in \Znn$. In particular, $\dssl{(\Hom{A}{G}{M})}=\LL{M}$.
\end{thm}
\begin{proof}
For $m$ satisfying $1 \leq m \leq \LL{M}$ we prove the claim by induction on $m$, starting with $m=1$. Note that $\Hom{A}{G}{\Soc{M}}$ is a direct sum of standard modules of type $\Hom{A}{G}{L_{i}}=\Delta\B{i,1}$, so $\Hom{A}{G}{\Soc{M}} \subseteq \delta\B{\Hom{A}{G}{M}}$. Since the functor $\Hom{A}{G}{-}$ preserves injective hulls (see \cite[Lemma $4.4$]{MR3510398}), the modules $\Hom{A}{G}{\Soc{M}}$ and $\Hom{A}{G}{M}$ have the same socle. Hence the previous inclusion must be an equality.

Suppose now that $2 \leq m \leq \LL{M}$, and set
%By induction, we have
%\begin{multline*}
%\delta_m \B{\Hom{A}{G}{M}}/ \Hom{A}{G}{\SOC{m-1}{M}} \\ 
%\begin{aligned} 
%&=\delta_m \B{\Hom{A}{G}{M}}/ \delta_{m-1} \B{\Hom{A}{G}{M}} \\ 
%& =\delta\B{\Hom{A}{G}{M}/ \delta_{m-1} \B{\Hom{A}{G}{M}}}\\
%&= \delta\B{\Hom{A}{G}{M}/\Hom{A}{G}{\SOC{m-1}{M}}}.
%\end{aligned}
%\end{multline*}
\begin{gather*}
Z_1:=\Hom{A}{G}{M}/ \Hom{A}{G}{\SOC{m-1}{M}} \\
Z_2:= \Hom{A}{G}{\SOC{m}{M}}/ \Hom{A}{G}{\SOC{m-1}{M}}.
\end{gather*}
Since $\Hom{A}{G}{-}$ preserves injective hulls, the modules $\Hom{A}{G}{M/\SOC{m-1}{M}}$ and $\Hom{A}{G}{\SOC{m}{M}/\SOC{m-1}{M}}$ have the same socle. But then, by Corollary \ref{cor:neweasyses}, $Z_1$ and $Z_2$ have the same socle. Moreover, $Z_1$ belongs to $\mathcal{F}\B{\Delta}$. By Lemma \ref{lem:soc0}, $Z_2$ must be contained in $\delta\B{Z_1}$. So both $\delta\B{Z_1}$ and $Z_2$ are $\Delta$-semisimple modules with the same socle. By Corollary \ref{cor:neweasyses}, $Z_1/Z_2$ must be in $\mathcal{F}\B{\Delta}$. Since $\mathcal{F}\B{\Delta}$ is closed under submodules, then $\delta\B{Z_1}/ Z_2$ is in $\mathcal{F}\B{\Delta}$. We must have $\delta\B{Z_1}/ Z_2=0$, otherwise this factor module would have some composition factor of type $L_{i,l_i}$. By induction, we may suppose that $\delta_{m-1}(\Hom{A}{G}{M})=\Hom{A}{G}{\SOC{m-1}{M}}$. Then, the identity $Z_2=\delta\B{Z_1}$ translates to
\begin{multline*}
\Hom{A}{G}{\SOC{m}{M}}/ \delta_{m-1}\B{\Hom{A}{G}{M}} \\
=\delta_{m}\B{\Hom{A}{G}{M}}/ \delta_{m-1}\B{\Hom{A}{G}{M}}.
\end{multline*}
This implies that $\delta_{m}\B{\Hom{A}{G}{M}}=\Hom{A}{G}{\SOC{m}{M}}$, $1 \leq m \leq \LL{M}$. The same identity holds trivially for $m=0$ and for $m \geq \LL{M}$.
\end{proof}

\section{Projective covers of modules over the ADR algebra}
\label{sec:thma}

We would like to determine the projective covers of modules over the ADR algebra $R_A$ of $A$. For a module $M$ in $\Mod{A}$, the projective cover $p_*$ of $\Hom{A}{G}{M}$ in $\Mod{R_A}$ is the image of an epic $p$, with domain in $\Add{G}$, through the functor $\Hom{A}{G}{-}$. The morphism $p$ is a special kind of map: it is the right minimal $\Add{G}$-approximation of $M$ in $\Mod{A}$.

The problem of finding approximations is hard in general. However, as we shall see in Theorem \ref{customthm:addGapprox}, it is very easy to compute right $\Add{G}$-approximations of rigid modules. Recall that a module is \emph{rigid} if its radical series coincides with its socle series.

Theorem \ref{customthm:addGapprox} (or rather consequences of this result -- Corollary \ref{cor:lemmaa} and Proposition \ref{prop:pink2}) will be very useful when dealing with examples. In Subsection \ref{subsec:counterexample}, we will use Corollary \ref{cor:lemmaa} and Proposition \ref{prop:pink2} to give a counterexample to a claim by Auslander, Platzeck and Todorov (\cite{MR1052903}) about the projective resolutions of modules over the ADR algebra.

\subsection{Approximations of rigid modules}
\label{subsec:approx}
Let $\mathcal{X}$ be a class of $A$-modules. We recall the definition of right $\mathcal{X}$-approximation and of right minimal morphism. A morphism $f: X \longrightarrow M$ in $\Mod{A}$, with $X$ in $\mathcal{X}$, is said to be a \emph{right $\mathcal{X}$-approximation} of $M$ if $\Hom{A}{X'}{f}$ is an epic for all $X'$ in $\mathcal{X}$. A map $f:M \longrightarrow N $ in $\Mod{A}$ is called a \emph{right minimal morphism} if every endomorphism $g: M \longrightarrow M$ satisfying $f=f \circ g$ is an automorphism.

The right $\Add{G}$-approximations of a module $M$ in $\Mod{A}$ are in bijection with epics in $\Mod{R_A}$,
\[
\begin{tikzcd}[ampersand replacement=\&]
\Hom{A}{G}{X} \arrow[two heads]{r} \& \Hom{A}{G}{M},
\end{tikzcd}
\]
where $X \in \Add{G}$. This bijection restricts to a one-to-one correspondence between right minimal $\Add{G}$-approximations in $\Mod{A}$ and projective covers in $\Mod{R_A}$. Since $G$ is a generator, the functor $\Hom{A}{G}{-}$ is particularly well behaved: it is fully faithful and it is such that the projective cover of a module $M$ in $\Mod{A}$ factors through its $\Add{G}$-approximation. The latter statement implies that every right $\Add{G}$-approximation is an epimorphism. 
%It is easy to see that $\Hom{A}{G}{-}$ maps bijectively right minimal $\Add{G'}$-approximations in $\Mod{A}$ to projective covers in $\Mod{R'}$. 
\begin{thm}
\label{customthm:addGapprox}
Let $M$ be a rigid module in $\Mod{A}$ such that $\LL{M}=m$. The projective cover of $M$ in $\Mod{(A/ \RAD{m}{A})}$ is a right minimal $\Add{G}$-approximation of $M$.
\end{thm}

\begin{proof}
Let $M$ be a rigid module with Loewy length $m$. Consider the projective cover of $M$ as an $(A/ \RAD{m}{A})$-module,
\[
\varepsilon: P_0\B{M} \longrightarrow M.
\]
We want to prove that $\varepsilon$ is a right minimal $\Add{G}$-approximation. By definition, $\varepsilon$ is a right minimal morphism, so it is enough to prove that every map $f:P_i/\RAD{j}{P_i} \longrightarrow M$, with $(i,j) \in \Lambda$, factors through $\varepsilon$. Note that this holds for $j \geq m$, as $\varepsilon$ is an epic in $\Mod{(A/\RAD{j}{A})}$ and $P_i/ \RAD{j}{P_i}$ is a projective $(A/ \RAD{j}{A})$-module. So suppose that $j < m$. Then
\[
\Ima{f} \subseteq \SOC{j}{M} = \RAD{m-j}{M},
\]
using that $M$ is rigid. Observe that both $\RAD{m-j}{M}$ and $\RAD{m-j}{(P_0\B{M})}$ are annihilated by $\RAD{j}{A}$, i.e.~they lie in $\Mod{(A/\RAD{j}{A})}$. Now note that the functor $\RAD{m-j}{\B{-}}$ preserves epics. This can be seen directly, or can be deduced by looking at Example \ref{ex:socrad} and Remark \ref{rem:dual-hereditary-cohereditary}, recalling that the composition of cohereditary preradicals is still a cohereditary preradical. Therefore we have the diagram
\[
\begin{tikzcd}[ampersand replacement=\&]
P_i/ \RAD{j}{P_i} \arrow{rr}{f}\arrow[dashed]{dd}{\exists\, t} \arrow[two heads]{rd} \& \& M \\
 \& \Ima{f} \arrow[hook]{d}\arrow[hook]{ur} \& \\
 \RAD{m-j}{(P_0\B{M})} \arrow[two heads]{r}{\RAD{m-j}{\varepsilon}} \& \RAD{m-j}{M} \arrow[hook]{ruu}[swap]{\iota_M}.
\end{tikzcd},
\]
where $t$ exists because $P_i/ \RAD{j}{P_i} $ is a projective in $\Mod{(A/\RAD{j}{A})}$. Thus
\[
f= \iota_M \circ (\RAD{m-j}{\varepsilon}) \circ t = \varepsilon \circ \iota_{P_0\B{M}} \circ t,
\]
where $\iota_{P_0\B{M}}$ denotes the inclusion of $\RAD{m-j}{(P_0\B{M})}$ in $P_0\B{M}$.
\end{proof}
As an immediate consequence of Theorem \ref{customthm:addGapprox}, we get the following result.
\begin{cor}
\label{cor:lemmaa}
Let $M$ be a rigid module in $\Mod{A}$ with $\LL{M}=m$. Suppose that $\varepsilon$ is the projective cover of $M$ in $\Mod{(A/ \RAD{m}{A})}$. Then $\Hom{A}{G}{\varepsilon}$ is the projective cover of $\Hom{A}{G}{M}$ in $\Mod{R_A}$.
\end{cor}

The simple modules over the ADR algebra $R_A$ are ``linked to each other" in a neat way. When all projective indecomposable modules are rigid then the `glueing' of the simple modules (and of the standard modules) is even nicer.
\begin{prop}
\label{prop:pink2}
Let $(i,j)$ and $(k,l)$ be in $\Lambda$. Then $\Ext{R_A}{1}{L_{i,j}}{L_{k,l}}\neq 0$ implies that either $(k,l)=(i,j+1)$ or $l \leq j-1$.  If the $A$-module $P_i/\RAD{j}{P_i}$ is rigid then $\Ext{R_A}{1}{L_{i,j}}{L_{k,l}}\neq 0$ implies that either $(k,l)=(i,j+1)$ or $l = j-1$. In particular, the latter statement holds when all the projective indecomposable $A$-modules are rigid.
\end{prop}
\begin{proof}
Observe that $\Ext{R_A}{1}{L_{i,j}}{L_{k,l}} \neq 0$ if and only if the simple module $L_{k,l}$ is a summand of $\Rad{P_{i,j}}/ \RAD{2}{P_{i,j}}$. The short exact sequence \eqref{eq:elreidsebastiao} in the statement of Theorem \ref{prop:standard} gives rise to the exact sequence
\[
\begin{tikzcd}[ampersand replacement=\&, column sep=scriptsize]
0 \arrow{r} \& \Hom{A}{G}{\Rad{P_i}/\RAD{j}{P_i}} \arrow{r} \& \Rad{P_{i,j}} \arrow{r} \& \Rad{\Delta\B{i,j}} \arrow{r} \& 0
\end{tikzcd},
\]
where $\Rad{\Delta\B{i,j}}=\Delta\B{i,j+1}$. If $L_{k,l}$ is a summand of the top of $\Rad{P_{i,j}}$ then either $(k,l)=(i,j+1)$ or $L_{k,l}$ is a summand of the top of $\Hom{A}{G}{\Rad{P_i}/\RAD{j}{P_i}}$. In the latter case, we must have $l \leq j-1$ by \cite[Proposition $10.2$]{MR0349747} (see also \cite[Lemma $3.3$]{MR3510398}).

If $P_i / \RAD{j}{P_i}$ is rigid, then $\Rad{P_i}/\RAD{j}{P_i}$ is also rigid. In this case, Corollary~\ref{cor:lemmaa} implies that the summands of the top of $\Hom{A}{G}{\Rad{P_i}/\RAD{j}{P_i}}$ are of type $L_{k,j-1}$.
\end{proof}

The next example shows that the rigidity condition in the statement of Theorem \ref{customthm:addGapprox} cannot be omitted. 
\begin{ex}
\label{ex:quiver}
Consider the quiver
\[
Q=
\begin{tikzcd}[ampersand replacement=\&]
 \& 1 \arrow{ld}{\alpha} \arrow{d}{\beta} \arrow{rd}{\gamma} \&\\
2 \& 3 \arrow{d}{\varepsilon} \& 4 \arrow{d}{\eta} \\
\& 5 \& 6
\end{tikzcd}
\]
and the path algebra $A=KQ$. Let $M$ be the $A$-module $P_1/L_6$, that is,
\[
M:=\begin{tikzcd}[ampersand replacement=\&, row sep =tiny, column sep = tiny]
 \& 1 \arrow[dash]{ld} \arrow[dash]{d} \arrow[dash]{rd} \&\\
2\& 3 \arrow[dash]{d} \& 4 \\
\& 5 \& 
\end{tikzcd}.
\]
Observe that $\LL{M}=\LL{A}=3$, and that $M$ is not a rigid module. 

Consider the epic $\pi: P_1 \longrightarrow M$ and note that the simple module $L_4$ can be embedded in $M$. It is not difficult to check that the epic
\begin{equation}
\label{eq:usingtheepic}
\begin{bmatrix}
\pi & 1_{L_4}
\end{bmatrix}:
P_1 \oplus L_4 \longrightarrow M
\end{equation}
is a right minimal $\Add{G}$-approximation of $M$. This map is not a projective cover of $M$, so the rigidity condition in the statement of Theorem \ref{customthm:addGapprox} is necessary. 

Using the approximation \eqref{eq:usingtheepic}, one easily sees that the $R_A$-module $\Hom{A}{G}{M}$ can be represented as
\[
\begin{tikzcd}[ampersand replacement=\&, row sep =tiny, column sep = tiny]
 \& (1,3) \arrow[dash]{ld} \arrow[dash]{d} \arrow[dash]{rd} \& (4,1)\arrow[dash]{d} \\
(2,1) \& (3,2) \arrow[dash]{d} \& (4,2) \\
\& (5,1) \& 
\end{tikzcd}.
\]
\end{ex}
\subsection{An application}
\label{subsec:counterexample}

Motivated in part by the theory of quasihereditary algebras, Auslander, Platzeck and Todorov studied in \cite{MR1052903} the homological properties of idempotent ideals. In this paper the authors defined a new class of algebras -- the Artin algebras satisfying the descending Loewy length condition -- and proved, in Theorem $7.3$, \cite{MR1052903}, that every such algebra is quasihereditary. 
\begin{defi}[{\cite[§7]{MR1052903}}]
\label{def:dllc}
An Artin algebra $B$ satisfies the \emph{descending Loewy length condition} (DLL condition, for short) if for every $M$ in $\Mod{B}$, a minimal projective resolution 
\[
\begin{tikzcd}[ampersand replacement=\&]
\cdots\arrow{r}
      \& P_{i}\B{M} \arrow{r}
        \& \cdots \arrow{r}
        \& P_0\B{M} \arrow{r}{\varepsilon}
        \& M \arrow{r}
            \& 0
\end{tikzcd}
\]
satisfies $\LL{P_{i+1}\B{M}} < \LL{P_{i}\B{M}}$, for all $i \geq 1$ such that $P_i\B{M} \neq 0$.
\end{defi}

In \cite{MR1052903} the authors claim that the Artin algebras of global dimension 2, the ADR algebras $R_A$, and the $l$-hereditary algebras (introduced in \cite{MR607166}) all satisfy the DLL condition. The main purpose of Theorem $7.3$ in \cite{MR1052903} was thus to give a unified proof of results in \cite{MR987824}, \cite{MR943793} and \cite{MR964453}, already established in the literature.

It is not difficult to check that Artin algebras of global dimension $2$ and that $l$-hereditary algebras do satisfy the DLL condition. Unfortunately, it is not true that the ADR algebra $R_A$ satisfies the DLL condition for every choice of $A$. As we shall see, the Loewy length of the projectives in a projective resolution in $\Mod{R_A}$ may increase by an arbitrarily large number.

In order to see this, consider the following example: define $A:= KQ/ I$, where $K$ is a field, $Q$ is the quiver
\[
Q=
\begin{tikzcd}[ampersand replacement=\&]
\overset{1}{\circ} \arrow[loop left]{}{\varepsilon} \arrow[bend left]{r}{\alpha_1}\& \overset{2}{\circ}\arrow[bend left]{l}{\beta_1} \arrow[bend left]{r}{\alpha_2}\&  \overset{3}{\circ} \arrow[bend left]{l}{\beta_2} 
\end{tikzcd}
\]
and $I$ is the admissible ideal
\[
I= \langle \alpha_2\alpha_1, \, \beta_1 \beta_2, \, \beta_2 \alpha_2 - \alpha_1 \beta_1, \, \varepsilon \beta_1, \, \alpha_1 \varepsilon, \, \varepsilon^n \rangle ,
\]
with $n >1$ fixed. 

We may represent the projective indecomposable $A$-modules as
\[
\begin{tikzcd}[ampersand replacement=\&, row sep =tiny, column sep = tiny]
\& 1 \arrow[dash]{dl} \arrow[dash]{dr} \& \\
1  \arrow[dash]{d} \& \& 2 \arrow[dash]{d} \\
1 \arrow[dash]{d} \& \& 1 \\
\vdots \arrow[dash]{d} \& \& \\
1  \& \&
\end{tikzcd} , \, \,
\begin{tikzcd}[ampersand replacement=\&, row sep =tiny, column sep = tiny]
 \& 2 \arrow[dash]{dl} \arrow[dash]{dr} \& \\
 1  \arrow[dash]{dr} \& \& 3 \arrow[dash]{dl} \\
 \& 2\&  
\end{tikzcd}, \,\,
\begin{tikzcd}[ampersand replacement=\&, row sep =tiny, column sep = tiny]
3 \arrow[dash]{d}
\\ 2  \arrow[dash]{d} \\
3
\end{tikzcd}.
\]

Note that the $A$-module $L_3$ is in the socle of $P_3$. Thus, using the labelling in \eqref{eq:posetadr}, the $R_A$-module $P_{3,3}$ contains a copy of $\Delta\B{3,1}$. The module $\Delta\B{3,1}$ has socle $L_{3,3}$, so we may consider the corresponding quotient module $M:=P_{3,3}/ L_{3,3}$.

\begin{prop}
\label{label:annedorme}
Let $A$ be the algebra introduced previously and consider the corresponding ADR algebra $R_A$. Let $M$ be the $R_A$-module defined above. The DLL condition fails for the $R_A$-module $M$ when $n \geq 5$. Indeed, we have $\LL{P_1\B{M}}\leq 6$ and $\LL{P_2\B{M}}\geq 1+n$, so $\LL{P_2\B{M}}\geq \LL{P_1\B{M}}$ for $n \geq 5$.
\end{prop}
\begin{proof}
Using that $\LL{P_3}=3$, together with Theorem \ref{prop:standard}, we conclude that $\Rad{P_{3,3}}=\Hom{A}{G}{\Rad{P_3}}$. Since $\Rad{P_3}$ is rigid, Corollary \ref{cor:lemmaa} implies that the minimal projective presentation of $L_{3,3}$ is of the form
\[
\begin{tikzcd}[ampersand replacement=\&]
P_{2,2} \arrow{r} \& P_{3,3} \arrow{r} \& L_{3,3} \arrow{r} \& 0.
\end{tikzcd}
\]
We claim that $\LL{P_{3,3}}\leq 6$ and $ \LL{P_{2,2}}\geq 1+n \geq 5$. Note that this will imply the statement in the proposition, as $P_{i+1}\B{M}=P_i\B{L_{3,3}}$.

We start by showing that $ \LL{P_{2,2}}\geq 1+n $. To see this, note that the $A$-module $L_1$ is in the socle of $P_2/\RAD{2}{P_2}$. Thus, the $R_A$-module $P_{2,2}$ contains a copy of $\Delta\B{1,1}$. Note that $\LL{\Delta\B{1,1}}=n$ (see Proposition \ref{prop:newprop}), so $\LL{P_{2,2}}\geq 1+n$ (actually this is an equality).

Using that $P_{3,3}=\Hom{A}{G}{P_3}$, we deduce through a routine computation that $P_{3,3}$ has dimension $6$. Consequently, $P_{3,3}$ has Jordan--H\"{o}lder length $6$ and $\LL{P_{3,3}} \leq 6$.
\end{proof}

We have just shown that the module $P_{3,3}$ has Loewy length at most 6. One can actually prove that $\LL{P_{3,3}}=5$ and, as a consequence, refine the statement of Proposition \ref{label:annedorme} for $n \geq 4$.

In order to emphasise the usefulness of the results in Subsection \ref{subsec:approx}, we compute the exact Loewy length of $P_{3,3}$.
\begin{lem}
The module $P_{3,3}$ has Loewy length equal to $5$.
\end{lem}
\begin{proof}
Observe that
\begin{align*}
\LL{P_{3,3}} &=1 + \LL{\Rad{P_{3,3}}} \\
& =1 + \LL{\Hom{A}{G}{\Rad{P_3}}} .
\end{align*}
Define $N_1:=\Hom{A}{G}{\Rad{P_3}}$. We claim that $\LL{N_1}=4$.

By Lemma \ref{lem:soc0} and Theorem \ref{thm:socdelta}, the module $N_1$ has a $\Delta$-semisimple filtration
\[
0 \subset \Hom{A}{G}{L_3} \subset N_1,
\]
with factors $\Delta\B{3,1}$ and $\Delta\B{2,2}$. In particular, $N_1$ has socle $L_{3,3}$. Consider now the pullback diagram
%\begin{equation}
%\label{eq:pullbackexample}
\[\begin{tikzcd}[ampersand replacement=\&]
0 \arrow{r}\& \Delta\B{3,1} \arrow{r}\arrow[equal]{d} \& N'_2\arrow{r}\arrow[hook]{d} \& L_{2,3} \arrow[hook]{d} \arrow{r}\& 0 \\
0 \arrow{r} \& \Delta\B{3,1} \arrow{r} \& N_1 \arrow{r} \& \Delta\B{2,2} \arrow{r} \& 0
\end{tikzcd}.
\]
There is a (unique) submodule $N_2$ of $N_1$ with $\Top{N_2}=L_{2,3}$. In fact, by looking at the diagram above we see that $N_2 \subseteq N'_2$, so
\[
\Rad{N_2} \subseteq \Rad{N_2'} \subseteq \Delta\B{3,1}.
\]
Note that $P_2$ is rigid. Using Proposition \ref{prop:pink2} (and the structure of $\Delta (3,1)$), we conclude that $\Rad{N_2}$ must have top $L_{3,2}$. Thus, $N_2$ has the following structure
\[N_2=
\begin{tikzcd}[ampersand replacement=\&, row sep =tiny, column sep = tiny]
(2,3) \arrow[dash]{d}
\\ (3,2) \arrow[dash]{d} \\
(3,3)
\end{tikzcd}.
\]
Recall that Lemma \ref{lem:hereditary} applies to $\Soc{\B{-}}$ and $\SOC{2}{\B{-}}$. The image of the monic $\Delta (3,1)\longrightarrow N_1$ through the functor $1/\Soc{(-)}$ gives rise to the inclusion $L_{3,2} \subseteq \SOC{2}{N_1}/\SOC{1}{N_2}$. By applying $1/\SOC{2}{(-)}$ to the monics $\Delta (3,1) \longrightarrow N_1$ and $N_2 \longrightarrow N_1$ we deduce that
\[L_{3,1} \oplus L_{2,3} \subseteq \SOC{3}{N_1}/\SOC{2}{N_1} .\]
Now observe that $\Rad{P_3}$ is a rigid module. Corollary \ref{cor:lemmaa} implies that $\Top{N_1}=L_{2,2}$. Since $N_1$ has exactly 5 composition factors (and we have looked at them all), we conclude that $\LL{N_1}=4$. This proves the result.
\end{proof}
\begin{rem}
In \cite{MR1216693}, Lin and Xi extended Dlab and Ringel's result in \cite{MR943793} to endomorphism algebras of semilocal modules. The authors noticed that this class of algebras (which contains the ADR algebra) does not generally satisfy the DLL condition (see Example 3 in \cite{MR1216693}).
\end{rem}
\begin{rem}
Although the DLL condition does not hold for the ADR algebra $R_A$ in general, $R_A$ satisfies a property similar to the DLL condition. The following was implicitly proved in \cite{MR0349747}, within the proof of Proposition~$10.2$. 

Let $M$ be in $\Mod{A}$ with $\LL{M}=m$, and let $\varepsilon:X \longrightarrow M$ be the right minimal $\Add{G_{A/\RAD{m}{A}}}$-approximation of $M$. Then $\varepsilon$ is the right minimal $\Add{G}$-approximation of $M$ and $\LL{\Ker{\varepsilon}} < m=\LL{X}$. 

As a consequence, the projective resolutions in $\Mod{R_A}$ come from exact sequences in $\Mod{A}$ whose Loewy length decreases strictly. To be precise, for every $N$ in $\Mod{R_A}$ there is an exact sequence of $A$-modules
\[
\begin{tikzcd}[ampersand replacement=\&]
  0 \arrow{r} \&
    X_t \arrow{r} \&
      \cdots \arrow{r} \&
        X_2 \arrow{r} \&
          X_1 \arrow{r} \&
            X_0
\end{tikzcd},\]
with $X_i$ in $\Add{G}$ satisfying $\LL{X_{i+1}} < \LL{X_i}$ for all $i \geq 1$, such that
\[
\begin{tikzcd}[ampersand replacement=\&]
  0 \arrow{r} \&
    \Hom{A}{G}{X_t} \arrow{r} \&
      \cdots \arrow{r} \&
            \Hom{A}{G}{X_0} \arrow{r}{\varepsilon} \&
   N \arrow{r} \& 0
\end{tikzcd}
\]
is a minimal projective resolution for $N$ (see §$3.3.2$ in \cite{thesis} for details). 
\end{rem}

\bibliographystyle{amsplain}
\bibliography{QHAlg}

\end{document}